\documentclass[12pt]{amsart}
\usepackage[cp1251]{inputenc}
\usepackage{amsmath,eucal,verbatim,amsopn,amsthm,amsfonts,amssymb,mathdots,mathrsfs,esint,mathtools,enumitem,mnsymbol,bbm,hhline,url}
\usepackage{amsmath,amssymb,latexsym}
\usepackage{amscd}
\usepackage[margin=1in]{geometry}
\usepackage[stable]{footmisc}
\usepackage[hyperfootnotes=false,linktocpage=true,colorlinks=true,allcolors=blue]{hyperref}
\theoremstyle{plain}
\newtheorem{theorem}{Theorem}
\newtheorem{proposition}[theorem]{Proposition}

\newtheorem{lemma}[theorem]{Lemma}
\newtheorem{claim}[theorem]{Claim}
\newtheorem{definition}[theorem]{Definition}

\newtheorem{remark}[theorem]{Remark}
\newtheorem{example}[theorem]{Example}

\numberwithin{equation}{section}

\def\R{\mathbb{R}}

\def\C{\mathbb{C}}
\def\Mat{\mathrm{Mat}}

\DeclareMathOperator{\tr}{tr}

\DeclareMathOperator{\Hom}{Hom}

\DeclareMathOperator{\diag}{diag}

\DeclareMathOperator{\Ind}{Ind}
\DeclareMathOperator{\GL}{GL}
\DeclareMathOperator{\SL}{SL}

\DeclareMathOperator{\Real}{Re}

\begin{document}
\title[Generalized doubling: $(k,c)$ models]{The generalized doubling method: $(k,c)$ models}
\author{Yuanqing Cai}
\author{Solomon Friedberg}
\author{Dmitry Gourevitch}
\author{Eyal Kaplan}
\address{Cai: Faculty of Mathematics and Physics, Institute of Science and Engineering, Kanazawa University, Kakumamachi, Kanazawa, Ishikawa, 920-1192, Japan}
\email{cai@se.kanazawa-u.ac.jp}
\address{Friedberg:  Department of Mathematics, Boston College, Chestnut Hill, MA 02467-3806, USA}
\email{solomon.friedberg@bc.edu}
\address{Dmitry Gourevitch: Faculty of Mathematics and Computer Science, Weizmann Institute of Science, POB 26, Rehovot 76100, Israel }
\email{dmitry.gourevitch@weizmann.ac.il}
\address{Kaplan: Department of Mathematics, Bar Ilan University, Ramat Gan 5290002, Israel}
\email{kaplaney@gmail.com}
\thanks{This research was supported by the ERC, StG grant number 637912 (Cai), 
by the JSPS KAKENHI grant number 19F19019 (Cai), by MEXT Leading Initiative for Excellent Young Researchers Grant Number JPMXS0320200394 (Cai), by the BSF, grant number 2012019 (Friedberg), by the NSF, grant numbers 1500977, 1801497 and 2100206 (Friedberg), by the ERC, StG grant number 637912 (Gourevitch), and by the Israel Science Foundation, grant numbers 376/21 and 421/17 (Kaplan).}
\begin{abstract}
One of the key ingredients in the recent construction of the generalized doubling method is a new class of models, called $(k,c)$ models, for local components of generalized Speh representations. We construct a family of $(k,c)$ representations, in a purely local setting, and discuss their realizations using inductive formulas. Our main result is a uniqueness theorem which is essential for the proof that
the generalized doubling integral is Eulerian.
\end{abstract}
\maketitle
\addtocontents{toc}{\protect\setcounter{tocdepth}{1}}
\section*{Introduction}\label{intro}

A model is a fundamental concept in representation theory and integral representations. Typically, a model
for a representation arises from an equivariant functional, which allows one to realize the representation
in a space of complex-valued functions with some natural geometric properties. The useful cases are when the functional is unique up to scaling, and the class of representations affording
the model is broad. One important example is the Whittaker model, which has had a profound impact on the
study of representations with a vast number of applications, perhaps most notably Shahidi's theory of local
coefficients.

Let $F$ be a local field of characteristic $0$.
In this short note we discuss a new class of models, $(k,c)$ models, for representations of $\GL_{kc}=\GL_{kc}(F)$, which first appeared in the construction of
the generalized doubling integral (\cite{CFGK2}) in the context of generalized Speh representations. Our main result:
Theorem~\ref{theorem:any rep is WSS}, is that the local generalized Speh representation (defined in \cite{Jac4}) of $\GL_{kc}$ corresponding to a unitary generic representation $\tau$ of $\GL_k$ admits a unique $(k,c)$ model. Our result is in fact stronger: We construct
a map $\rho_c$ from irreducible generic representations $\tau$ of $\GL_k$ to $(k,c)$ representations $\rho_c(\tau)$.
This general context is essential for the analysis of the local generalized doubling integrals (\cite{CFGK2,DimaKaplan}) when data are ramified or archimedean.
We also discuss two realizations of the $(k,c)$ functional, which are also important for the study of such integrals.

The main application of Theorem~\ref{theorem:any rep is WSS} concerns the aforementioned integrals from \cite{CFGK2}:
This theorem completes the proof that the global integral of \cite{CFGK2} is Eulerian. In \textit{loc. cit.} uniqueness was
only proved when data are unramified, thus producing an ``almost Eulerian" integral, i.e., only separating out the unramified places (cf. \cite[(3.1)]{CFGK2}). The existence of an Euler product is important for the development of the local theory.

\section{Preliminaries}\label{Groups}
Let $F$ be a local field of characteristic $0$. Identify linear algebraic $F$-groups with their $F$ points, i.e., $\GL_l=\GL_l(F)$. Fix the Borel subgroup $B_{\GL_l}=T_{\GL_l}\ltimes N_{\GL_l}$ of upper triangular invertible matrices in $\GL_l$, where $T_{\GL_l}$ is the diagonal torus. For a $d$ parts composition $\beta=(\beta_1,\ldots,\beta_d)$ of $l$, $P_{\beta}=M_{\beta}\ltimes V_{\beta}$ denotes the corresponding standard parabolic subgroup, where $V_{\beta}<N_{\GL_l}$. The unipotent subgroup opposite to $V_{\beta}$ is denoted $V_{\beta}^-$ and $\delta_{P_{\beta}}$ is the modulus character of $P_{\beta}$. For an integer $c\geq0$, $\beta c=(\beta_1c,\ldots,\beta_dc)$ is a composition of $lc$. Let $w_{\beta}$ be the permutation matrix consisting of blocks of identity matrices $I_{\beta_1},\ldots, I_{\beta_d}$, with $I_{\beta_i}\in\GL_{\beta_i}$ on its anti-diagonal, beginning with $I_{\beta_1}$ on the top right, then $I_{\beta_2}$, etc. E.g., $w_{(a,b)}=\left(\begin{smallmatrix}&I_a\\I_b\end{smallmatrix}\right)$.
We use $\tau_{\beta}$ to denote a representation of $M_{\beta}$, where $\tau_{\beta}=\otimes_{i=1}^d\tau_i$ ($\tau_i$ is then a
representation of $\GL_{\beta_i}$). Let $\Mat_{a\times b}$ and $\Mat_{a}$ denote the spaces of $a\times b$ or $a\times a$ matrices.
For $g\in \Mat_{a\times b}$, ${}^tg$ is the transpose of $g$. The trace map is denoted $\tr$. For $x,y\in \GL_l$, ${}^{x}y=xyx^{-1}$ and if $Y<\GL_l$, ${}^xY=\{{}^xy:y\in Y\}$.

All representations in this work are by definition complex and smooth. A generic representation of $\GL_l$ will be admissible, by definition. Over archimedean fields, by an admissible representation we mean admissible Fr\'{e}chet of moderate growth. We use the smooth and normalized induction functor.

Let $U<R<\GL_l$ be closed subgroups such that $U$ is a unipotent subgroup, and fix a character $\psi$ of $U$. For a representation $\sigma$ of $R$ on a space $\mathcal{V}$, the Jacquet module $J_{U,\psi}(\pi)$ is the quotient of $\mathcal{V}$ by the subspace spanned by $\{\pi(u)\xi-\psi(u)\xi:\xi\in\mathcal{V},u\in U\}$ over $p$-adic fields, and by the closure of this subspace for archimedean fields.
Then $J_{U,\psi}(\pi)$ is a representation of $R$ and we normalize the action as in \cite[1.8]{BZ2}.

When the field is $p$-adic, an entire (resp., meromorphic) function $f(\zeta_1,\ldots,\zeta_m):\C^m\rightarrow\C$ will always be an element of
$\C[q^{\mp\zeta_1},\ldots ,q^{\mp\zeta_m}]$, (resp., $\C(q^{-\zeta_1},\ldots, q^{-\zeta_m})$).

\section{Representations of type $(k,c)$}\label{Representations of type (k,c)}

\subsection{Definition}\label{(k,c) definition}
Let $k,c\geq1$ be integers. A partition $\sigma=(a_1,\ldots,a_l)$ of $kc$ such that $a_i>0$ for all $i$ identifies a
subgroup $V(\sigma)<N_{\GL_{kc}}$ as follows. Consider the multi-set of integers
$\Lambda_{\sigma}=\{a_i-2j+1:1\leq i \leq l,\,1\leq j\leq a_i\}$ and let $p_{\sigma}$ be the $kc$-tuple obtained by arranging $\Lambda_{\sigma}$ in decreasing order. For any $x\in F^*$, put $x^{p_{\sigma}}=\diag(x^{p_{\sigma}(1)},\ldots,x^{p_{\sigma}(kc)})\in T_{\GL_{kc}}$. The one-parameter subgroup
$\{x^{p_{\sigma}}:x\in F^*\}$ acts on the Lie algebra of $N_{\GL_{kc}}$ by conjugation, and $V(\sigma)$ is the subgroup generated by the weight subspaces of weight at least $2$. (This is not the subgroup $V_{\beta}$ defined for compositions). Let
$G_{\sigma}<\GL_{kc}$ denote the centralizer of $\{x^{p_{\sigma}}:x\in F^*\}$, it acts on the set of characters of $V(\sigma)$. Under
this action there is a unique open orbit $\mathcal{O}$ over $\C$, let $\psi_o\in \mathcal{O}$. A character $\psi'$ of $V(\sigma)$ is called generic if its stabilizer $G_{\sigma,\psi}$ in $G_{\sigma}$ is of the same type as $G_{\sigma,\psi_o}$ over $\C$. For further details see
\cite{G2}, \cite[\S~5]{CM} and \cite{Cr}. Let $\widehat{V}(\sigma)_{\mathrm{gen}}$ denote the set of generic characters of $V(\sigma)$. If $\sigma'$ is another partition of $kc$, write $\sigma'\succsim\sigma$ if $\sigma'$ is greater than or non-comparable with $\sigma$ under the natural partial ordering.

For $\sigma=(k^c)$, $V(\sigma)=V_{(c^k)}$ and $M_{(c^k)}$ acts transitively on $\widehat{V}(\sigma)_{\mathrm{gen}}$. We fix $\psi\in \widehat{V}(\sigma)_{\mathrm{gen}}$ by taking a nontrivial additive character $\psi$ of $F$ and extending it to a character of $V_{(c^k)}$ by
\begin{align}\label{eq:k c character}
\psi(v)=\psi(\sum_{i=1}^{k-1}\mathrm{tr}(v_{i,i+1})),\qquad v=(v_{i,j})_{1\leq i,j\leq k},\qquad v_{i,j}\in\Mat_c.
\end{align}
Then $G_{\sigma,\psi}=\GL_c^{\triangle}$, the diagonal embedding of $\GL_c$ in $M_{(c^k)}$.

\begin{definition}
An admissible representation $\rho$ of $\GL_{kc}$ is of type $(k,c)$ (briefly, $\rho$ is $(k,c)$), if $\Hom_{V(\sigma)}(\rho,\psi')=0$ for all $\sigma\succsim(k^c)$ and $\psi'\in\widehat{V}(\sigma)_{\mathrm{gen}}$, and
$\dim\Hom_{V_{(c^k)}}(\rho,\psi)=1$.
\end{definition}

A $(k,c)$ functional on $\rho$ (with respect to $\psi$) is a nonzero element of $\Hom_{V_{(c^k)}}(\rho,\psi)$. If $\rho$ is of type $(k,c)$, the space of such functionals is one dimensional. The resulting model (which is unique by definition) is called a $(k,c)$ model, and denoted $W_{\psi}(\rho)$. If $\lambda$ is a fixed $(k,c)$ functional, $W_{\psi}(\rho)$ is the space of functions $g\mapsto\lambda(\rho(g)\xi)$ where $g\in\GL_{kc}$ and $\xi$ is a vector in the space of $\rho$. Then $W_{\psi}(\rho)$ is a quotient of $\rho$, and when $\rho$ is irreducible $W_{\psi}(\rho)\cong\rho$.

The following is a Rodier-type result for $(k,c)$ representations.
\begin{proposition}\label{prop:rodier non}
For $1\leq i\leq d$, let $\rho_i$ be a $(k_i,c)$ representation. If $F$ is archimedean we further assume $k_1=\ldots=k_d=1$ and for each $i$, $\rho_i=\tau_i\circ\det$ for a quasi-character $\tau_i$ of $F^*$ and $\det$ defined on $\GL_c$.
Then $\rho=\Ind_{P_{\beta c}}^{\GL_{kc}}(\otimes_{i=1}^d\rho_i)$ is of type $(k,c)$, where $\beta=(k_1,\ldots,k_d)$ and $k=k_1+\ldots+k_d$ (if $F$ is archimedean, $k=d$ and $\beta=(1^k)$).
\end{proposition}
\begin{proof}
We need to prove $\Hom_{V(\beta')}(\rho,\psi')=0$ for any partition $\beta'\succsim(k^c)$ and character $\psi'\in \widehat{V}(\beta')_{\mathrm{gen}}$, and $\dim\Hom_{V_{(c^k)}}(\rho,\psi)=1$.

We will use the theory of detivatives of Bernstein and Zelevinsky \cite{BZ1,BZ2} over $p$-adic fields, its partial extension to archimedean fields by Aizenbud \textit{et. al.} \cite{AGS2015,AGS2015b}, and the relation between derivatives and degenerate Whittaker models developed (over both fields) by Gomez \textit{et. al.} \cite{GGS}.
Let $P_l$ be the subgroup of matrices $g\in\GL_l$ with the last row $(0,\ldots,0,1)$ ($P_l<P_{(l-1,1)}$) and let $\psi_l$ be the character of $V_{(l-1,1)}$ given by $\psi_l(\left(\begin{smallmatrix}I_{l-1}&v\\&1\end{smallmatrix}\right))=\psi(v_{l-1})$.
Then we have the functor $\Phi^-$ from (smooth) representations of $P_l$ to representations of $P_{l-1}$ given by $\Phi^-(\varrho)=J_{V_{(l-1,1)},\psi_l}(\varrho)$.
For $0<r\leq l$, the $r$-th derivative of a representation $\varrho$ of $\GL_l$ is defined over $p$-adic fields by
$\varrho^{(r)}=J_{V_{(l-1,1)}}((\Phi^-)^{r-1}(\varrho|_{P_l}))$, and over archimedean fields by
$\varrho^{(r)}=((\Phi^-)^{r-1}(\varrho|_{P_l}))|_{\GL_{n-r}}$ (more precisely this is the pre-derivative, we use the term derivative for uniformity). Also $\varrho^{(0)}=\varrho$. The highest derivative of $\varrho$ is the representation $\varrho^{(r_0)}$ such that $\varrho^{(r_0)}\ne0$ and $\varrho^{(r)}=0$ for all $r>r_0$.

According to the definition of $(k_i,c)$ representations and \cite[Theorems E, F]{GGS} (which also apply over $p$-adic fields),
the highest derivative of $\rho_i$ is $\rho_i^{(k_i)}$. Then
the highest derivative of $\rho$ is $\rho^{(k)}$ and
$\rho^{(k)}=\Ind_{P_{\beta(c-1)}}^{\GL_{k(c-1)}}(\otimes_{i=1}^d\rho_i^{(k_i)})$,
where over archimedean fields $\rho_i^{(k_i)}=\rho_i^{(1)}=\tau_i\circ\det$ and $\det$ is the determinant of $\GL_{c-1}$.
In the $p$-adic case this follows from \cite[Lemma~4.5]{BZ2}; in the archimedean case this follows from
\cite[Corollary~2.4.4]{AGS2015} and \cite[Theorem~B]{AGS2015b} (see also \cite[\S~4.4]{GGS}).
Now the highest derivative of $\rho^{(k)}$ is again its $k$-th derivative and we can repeat this process $c$ times
to obtain a one dimensional vector space. We conclude from \cite[Theorems E, F]{GGS} that $\rho$ admits a unique $(k,c)$ model.

If $\lambda\succsim(k^c)$, we can assume $\lambda_1>k$ then $\rho^{(k+1)}=0$, by \cite[Lemma~4.5]{BZ2} and
\cite[Theorem~B]{AGS2015b}. By \cite[Theorems E, F]{GGS}, this proves the required vanishing properties.
\end{proof}
\begin{remark}
The result for $p$-adic fields is stronger, because we have a general ``Leibniz rule" for derivatives (\cite[Lemma~4.5]{BZ2}).
\end{remark}

\subsection{The representation $\rho_c(\tau)$}\label{the representation rho_c(tau)}
Let $\tau$ be an irreducible generic representation of $\GL_k$ (for $k=1$, this means $\tau$ is a quasi-character of $F^*$). It is of type $(k,1)$, by the uniqueness of Whittaker models and because $(k)$ is the maximal unipotent orbit for $\GL_k$.
For any $c$, we construct a $(k,c)$ representation $\rho_c(\tau)$ as follows.
First assume $\tau$ is unitary.
For $\zeta\in\C^c$, consider the intertwining operator
\begin{align*}
M(\zeta,w_{(k^c)}):\Ind_{P_{(k^c)}}^{\GL_{kc}}(\otimes_{i=1}^c|\det|^{\zeta_i}\tau)\rightarrow\Ind_{P_{(k^c)}}^{\GL_{kc}}(\otimes_{i=1}^c|\det|^{\zeta_{c-i+1}}\tau).
\end{align*}
Given a section $\xi$ of $\Ind_{P_{(k^c)}}^{\GL_{kc}}(\otimes_{i=1}^c|\det|^{\zeta_i}\tau)$ which is a holomorphic function of $\zeta$,
$M(\zeta,w_{(k^c)})\xi$ is defined for $\Real(\zeta)$ in a suitable cone by the absolutely convergent integral
\begin{align*}
M(\zeta,w_{(k^c)})\xi(\zeta,g)=\int\limits_{V_{(k^c)}}\xi(\zeta,w_{(k^c)}^{-1}vg)\,dv,
\end{align*}
then by meromorphic continuation to $\C^c$.
Let $\rho_c(\tau)$ be the image of this operator at
\begin{align*}
\zeta=((c-1)/2,(c-3)/2,\ldots,(1-c)/2).
\end{align*}
Since $\tau$ is unitary, this image is well defined and irreducible by
Jacquet \cite[Proposition~2.2]{Jac4} (see also \cite[\S~I.11]{MW4}) and $\rho_c(\tau)$ is the unique irreducible quotient of
\begin{align}\label{rep:rho c tau as a quotient}
\Ind_{P_{(k^c)}}^{\GL_{kc}}((\tau\otimes\ldots\otimes\tau)\delta_{P_{(k^c)}}^{1/(2k)})
\end{align}
and the unique irreducible subrepresentation of
\begin{align}\label{rep:rho c tau as a subrep}
\Ind_{P_{(k^c)}}^{\GL_{kc}}((\tau\otimes\ldots\otimes\tau)\delta_{P_{(k^c)}}^{-1/(2k)}).
\end{align}

Now assume $\tau$ is an arbitrary irreducible generic representation of $\GL_k$. Then
$\tau=\Ind_{P_{\beta}}^{\GL_k}(\otimes_{i=1}^d|\det|^{a_i}\tau_i)$ where $\tau_i$ are tempered and $a_1>\ldots>a_d$
(by Langlands' classification and \cite{Vog78,Z3,JS}). Define
\begin{align}\label{rep:rho c tau in general}
\rho_c(\tau)=\Ind_{P_{\beta c}}^{\GL_{kc}}(\otimes_{i=1}^d|\det|^{a_i}\rho_c(\tau_i)).
\end{align}
Clearly $\rho_c(|\det|^{s_0}\tau)=|\det|^{s_0}\rho_c(\tau)$ for any $s_0\in\C$.

Note that when $\tau$ is unitary, the definition as the image of an intertwining operator agrees with the definition \eqref{rep:rho c tau in general}. Indeed in this case by Tadi{\'c} \cite{tadic86} and Vogan \cite{Vog86} we have
$\tau\cong\Ind_{P_{\beta'}}^{\GL_k}(\otimes_{i=1}^{d'}|\det|^{r_i}\tau_i')$, where $\tau_i'$ are square-integrable and $1/2>r_1\geq \ldots\geq r_{d'}>-1/2$. Then by \cite[\S~I.11]{MW4} we can permute the blocks in $\rho_c(\tau)$ corresponding to different representations $\tau'_i$ hence
\begin{align}\label{q:rho c tau for unitary in terms of square integrable}
\rho_c(\tau)\cong\Ind_{P_{\beta' c}}^{\GL_{kc}}(\otimes_{i=1}^{d'}|\det|^{r_i}\rho_c(\tau_i')).
\end{align}
In particular, e.g.,
\begin{align}\label{eq:rho c and ind rho c}
\rho_c(\Ind_{P_{(\beta'_1,\beta'_2)}}^{\GL_{\beta'_1+\beta'_2}}(\tau_1'\otimes\tau_2'))\cong
\Ind_{P_{(\beta'_1c,\beta'_2c)}}^{\GL_{(\beta'_1+\beta'_2)c}}(\rho_c(\tau_1')\otimes\rho_c(\tau_2')),
\end{align}
where the l.h.s.~ (left hand side) is defined as the image of the intertwining operator. Hence
if $a_1>\ldots>a_d$ are the $d\leq d'$ distinct numbers among $r_1,\ldots,r_{d'}$, $\tau_i$ is the tempered representation parabolically induced from $\otimes_{1\leq j\leq d'\,:\,r_j=a_i}\tau'_j$ and $\beta=(a_1,\ldots,a_d)$,
\begin{align*}
\rho_c(\tau)\cong\Ind_{P_{\beta c}}^{\GL_{kc}}(\otimes_{i=1}^{d}|\det|^{a_i}\rho_c(\tau_i)),
\end{align*}
which agrees with \eqref{rep:rho c tau in general}.
Moreover, \eqref{eq:rho c and ind rho c} implies that \eqref{rep:rho c tau in general} also holds when $a_1\geq\ldots\geq a_d$.

For example, if $\tau$ is irreducible unramified tempered, $\tau=\Ind_{B_{\GL_k}}^{\GL_k}(\otimes_{i=1}^k\tau_i)$ for
unramified unitary characters $\tau_i$ of $F^*$. By definition $\rho_c(\tau)$ is the unique irreducible unramified quotient of
\eqref{rep:rho c tau as a quotient}, but by \cite[\S~I.11]{MW4}, $\rho_c(\tau)=\Ind_{P_{(c^k)}}^{\GL_{kc}}(\otimes_{i=1}^k\tau_i\circ\det)$.

We extend the definition to certain unramified principal series, which are not necessarily irreducible. Assume
$\tau=\Ind_{B_{\GL_k}}^{\GL_k}(\otimes_{i=1}^k|~|^{a_i}\tau_i)$ ($k>1$) where $\tau_i$ are as above (e.g., unitary) but
$a_1\geq\ldots\geq a_k$ ($\tau$ is not a general unramified principal series because of the order). In this case $\tau$ still admits a unique Whittaker functional. We define $\rho_c(\tau)=\Ind_{P_{(c^k)}}^{\GL_{kc}}(\otimes_{i=1}^k|~|^{a_i}\tau_i\circ\det)$, which is a $(k,c)$ representation by Proposition~\ref{prop:rodier non}. Transitivity of induction and the example in the last paragraph imply that this definition coincides with \eqref{rep:rho c tau in general}, when $\tau$ is irreducible (in which case the order of $a_i$ does not matter).

\begin{theorem}\label{theorem:any rep is WSS}
Let $\tau$ be an irreducible generic representation of $\GL_k$.
Then $\rho_c(\tau)$ is $(k,c)$.
\end{theorem}
\begin{proof}
First assume $F$ is non-archimedean. We start by proving the result for square-integrable representations $\tau$.
By Zelevinsky \cite{Z3}, $\tau$ can be described as the unique irreducible subrepresentation of \begin{align}\label{rep:square integrable tau}
\Ind_{P_{(r^d)}}^{\GL_{k}}((\tau_0\otimes\ldots\otimes\tau_0)\delta_{P_{(r^d)}}^{1/(2r)}),
\end{align}
where $\tau_0$ is an irreducible unitary supercuspidal representation of $\GL_r$. Then by Tadi{\'c} (\cite[\S~6.1]{Tadic1986})
the highest Bernstein--Zelevinsky derivative of $\rho_c(\tau)$ is its $k$-th derivative and equals $|\det|^{-1/2}\rho_{c-1}(\tau)$
(see also \cite{Z3,tadic86} and \cite[Theorem 14]{LMinq2014}).
Repeatedly taking highest derivatives, we see that $\rho_c(\tau)$ is supported on $(k^c)$. The uniqueness of the functional follows as in the proof of Proposition \ref{prop:rodier non}. The proof for an arbitrary (irreducible generic) $\tau$ now follows from
\eqref{rep:rho c tau in general} and Proposition \ref{prop:rodier non}.

Now consider an archimedean $F$, and an irreducible generic $\tau$. For a smooth representation $\vartheta$ of $\GL_l$ let $\mathcal{V}(\vartheta)$ denote its annihilator variety and $WF(\vartheta)$ be its wave-front set (see e.g., \cite{GourevitchSahi2013} for these notions). According to the results of Vogan \cite{David1991} and Schmid and Vilonen \cite{SchmidVilonen2000}, $\mathcal{V}(\vartheta)$ is the Zariski closure of $WF(\vartheta)$
(see \cite[\S~3.3.1]{GGS}).
Applying this to $\vartheta=\rho_c(\tau)$, by \cite[Corollary 2.1.8]{GourevitchSahi2013} and \cite[Theorem~3]{SahiStein1990}
(see \cite[\S~4.2]{GourevitchSahi2013}), $(k_1+\ldots+k_d)^c=(k^c)$ is the maximal orbit in $WF(\rho_c(\tau))$. Note that this result holds although $\rho_c(\tau)$ may be reducible (associated cycles are additive, see e.g., \cite[p.~323]{David1991}). Therefore $\Hom_{V_{(\beta')}}(\rho_c(\tau),\psi')=0$ for all $\beta'$ greater than or not comparable with $(k^c)$ and generic character $\psi'$ of $V_{(\beta')}$, and also \cite[Theorem~E]{GGS} implies that $\rho_c(\tau)$ admits a $(k,c)$ functional. It remains
to show $\dim\Hom_{V_{(\beta')}}(\rho_c(\tau),\psi')\leq1$.

To this end, by \cite{BarbaschSahiSpeh1990,SahiStein1990} (see also \cite[Corollary~4.2.5]{GourevitchSahi2013}), each $\rho_c(\tau_i)$ with $\beta_i>1$ is a quotient of
$\Ind_{P_{(c^2)}}^{\GL_{2c}}(\rho_c(\chi_1)\otimes\rho_c(\chi_2))$ for suitable quasi-characters $\chi_1,\chi_2$ of $F^*$, hence
$\rho_c(\tau)$ itself is also a quotient of a degenerate principal series
$\Ind_{P_{(c^k)}}^{\GL_{kc}}(\otimes_{i=1}^k\rho_c(\chi_i))$. The latter is $(k,c)$ by Proposition \ref{prop:rodier non}, thus
$\dim\Hom_{V_{(\beta')}}(\rho_c(\tau),\psi')\leq1$. We deduce that $\rho_c(\tau)$ is $(k,c)$.
\end{proof}
\begin{remark}
As explained in the introduction, this theorem completes the proof that the global integrals from \cite{CFGK2}
are Eulerian.
\end{remark}
For an admissible representation $\varrho$ of $\GL_{l}$, let $\varrho^*(g)=\varrho(J_{l}{}^tg^{-1}J_{l})$ where
$J_l=w_{(1^l)}$. If $\varrho$ is irreducible, $\varrho^*\cong\varrho^{\vee}$.
\begin{claim}\label{claim:rho c dual}
If $\tau$ is tempered, $\rho_c(\tau)^{\vee}=\rho_c(\tau^{\vee})$. In general if $\tau$ is irreducible,
$\rho_c(\tau)^*=\rho_c(\tau^{\vee})$.
\end{claim}
\begin{proof}
The first assertion follows because $\rho_c(\tau)$ is a quotient of \eqref{rep:rho c tau as a quotient}, hence
both $\rho_c(\tau)^{\vee}$ and $\rho_c(\tau^{\vee})$ are irreducible subrepresentations of
$\Ind_{P_{(k^c)}}^{\GL_{kc}}((\tau^{\vee}\otimes\ldots\otimes\tau^{\vee})\delta_{P_{(k^c)}}^{-1/(2k)})$, but there is a unique such. The general case follows from
the definition, the tempered case and the fact that for any composition $\beta$ of $l$,
$(\Ind_{P_{\beta}}^{\GL_l}(\otimes_{i=1}^d\varrho_{i}))^*=\Ind_{P_{(\beta_d,\ldots,\beta_1)}}^{\GL_l}(\otimes_{i=1}^d\varrho_{d-i+1}^*)$.
\end{proof}
While \eqref{rep:rho c tau in general} may be reducible, it is still of finite length and
admits a central character.
We mention that since the Jacquet functor is exact over non-archimedean fields, and the generalized Whittaker functor of
\cite{GGS} is exact over archimedean fields (\cite[Corollary~G]{GGS}), $\rho_c(\tau)$ admits a unique irreducible subquotient which is a $(k,c)$ representation.

\section{Realizations of $(k,c)$ functionals}
\subsection{Explicit $(k,c)$ functionals from compositions of $k$}\label{k c functionals from partitions of k}
Let $\tau$ be an irreducible generic representation of $\GL_k$, where $k>1$. If $\tau$ is not supercuspidal, it is a quotient of
$\Ind_{P_{\beta}}^{\GL_k}(\tau_{\beta})$ for a nontrivial composition $\beta$ of $k$ and an irreducible generic representation $\tau_{\beta}$ (e.g., one can take $\tau_{\beta}$ to be supercuspidal). A standard technique for realizing the Whittaker model of $\tau$ is to write down the Jacquet integral on the induced representation (this integral stabilizes in the $p$-adic case). Since $\Ind_{P_{\beta}}^{\GL_k}(\tau_{\beta})$ also affords a unique Whittaker model, the functional on the induced representation factors through $\tau$. One may also twist the inducing data $\tau_{\beta}$ using auxiliary complex parameters, to obtain an absolutely convergent integral which admits an analytic continuation in these parameters. See e.g., \cite{Sh2,JPSS,Soudry,Soudry2}.

We generalize this idea to some extent, for $(k,c)$ functionals.
\begin{lemma}\label{lemma:rho kc in tau1 tau2}
If $\tau=\Ind_{P_{\beta}}^{\GL_k}(\tau_{\beta})$ with $\tau_{\beta}=\otimes_{i=1}^d\tau_i$, $\tau_i=|\det|^{a_i}\tau_{0,i}$, $a_1\geq\ldots\geq a_d$ and each $\tau_{0,i}$ is square-integrable, or $\tau$ is the essentially square-integrable quotient of $\Ind_{P_{\beta}}^{\GL_k}(\tau_{\beta})$ and $\tau_{\beta}$ is irreducible supercuspidal (this includes the case $\beta=(1^k)$, i.e., $\tau_{\beta}$ is a character of $T_{\GL_k}$, over any local field), then $\rho_c(\tau)$ is a quotient of $\Ind_{P_{\beta c}}^{\GL_{kc}}(\otimes_{i}\rho_c(\tau_i))$.
\end{lemma}
\begin{proof}
In the first case, this is true by \eqref{rep:rho c tau in general} and \eqref{eq:rho c and ind rho c}. For the essentially square-integrable case, over an archimedean field the result follows from \cite[Corollary~4.2.5]{GourevitchSahi2013} (see the proof of Theorem~\ref{theorem:any rep is WSS}), and if $F$ is $p$-adic from \cite[Theorem~7.1]{tadic86}.
\end{proof}
Take $\beta$ as in the lemma. If $F$ is archimedean assume $\beta=(1^k)$, i.e., $\rho_c(\tau)$ is a quotient of a degenerate principal series, which is always possible (see the proof of Theorem~\ref{theorem:any rep is WSS}). Denote $\beta=(\beta_1,\ldots,\beta_d)$, $\beta'=(\beta_d,\ldots,\beta_1)$ and consider the following Jacquet integral
\begin{align}\label{eq:(k,c)  functional using an integral}
\int\limits_{V_{\beta'c}}\xi(w_{\beta c}v)\psi^{-1}(v)\,dv,
\end{align}
where $\xi$ lies in the space of $\Ind_{P_{\beta c}}^{\GL_{kc}}(\otimes_{i=1}^dW_{\psi}(\rho_c(\tau_i)))$ and regarded as a complex-valued function, and $\psi$ is the restriction of \eqref{eq:k c character} to $V_{\beta'c}$. The integral \eqref{eq:(k,c)  functional using an integral} is formally a $(k,c)$ functional on the full induced space. Twist the inducing data and induce from $\otimes_{i=1}^d|\det|^{\zeta_i}\rho_c(\tau_i)$, for fixed $\zeta=(\zeta_1,\ldots,\zeta_d)\in\C^d$ with $\Real(\zeta_i-\zeta_{i+1})\gg0$ for all $i<d$. (Note that $W_{\psi}(|\det|^{\zeta_i}\rho_c(\tau_i))=|\det|^{\zeta_i}W_{\psi}(\rho_c(\tau_i))$).

Both $\Ind_{P_{\beta c}}^{\GL_{kc}}(\otimes_{i=1}^d\rho_c(\tau_i))$ and $\Ind_{P_{\beta c}}^{\GL_{kc}}(\otimes_{i=1}^d|\det|^{\zeta_i}\rho_c(\tau_i))$
are $(k,c)$ representations, by Proposition~\ref{prop:rodier non}. The condition on $\zeta$ implies that the integral is absolutely convergent for all $\xi$ (see e.g., \cite[Lemma~2.1]{Soudry2}). If we let $\zeta$ vary and $\xi$ is analytic in $\zeta$, the integral admits analytic continuation, which over archimedean fields is continuous in the data $\xi$. Over $p$-adic fields this follows from Bernstein's continuation principle (the corollary in \cite[\S~1]{Banks}), since we have uniqueness and the integral can be made constant. Over archimedean fields this follows from Wallach \cite{Wall88b,Wallach2006} ($P_{(c^k)}$ is ``very nice" and we induce from a degenerate principal series).
For any $\zeta_0\in\C^d$ one can choose data such that the continuation of \eqref{eq:(k,c)  functional using an integral} is nonzero at $\zeta=\zeta_0$. Taking $\zeta_0=0$ we obtain a $(k,c)$ functional, which is unique (up to scaling). Hence this functional factors through $\rho_c(\tau)$ and provides a realization of $W_{\psi}(\rho_c(\tau))$.

\subsection{Explicit $(k,c)$ functionals from compositions of $c$}\label{k c functionals from partitions of c}
Let $\tau$ be an irreducible generic representation of $\GL_k$, and assume an unramified twist of $\tau$ is unitary.
In this section we construct $(k,c)$ functionals on $\rho_c(\tau)$ using compositions of $c$.
Fix $0<l<c$. Since now both $\rho_l(\tau)$ and $\rho_{c-l}(\tau)$ embed in the corresponding spaces
\eqref{rep:rho c tau as a subrep}, $\rho_c(\tau)$ is a subrepresentation of
\begin{align}\label{rep:induced space 2 k c}
\Ind_{P_{(kl,k(c-l))}}^{\GL_{kc}}((W_{\psi}(\rho_l(\tau))\otimes W_{\psi}(\rho_{c-l}(\tau)))\delta_{P_{(kl,k(c-l))}}^{-1/(2k)}).
\end{align}
Both $(k,l)$ and $(k,c-l)$ models exist by Theorem~\ref{theorem:any rep is WSS}.
We may regard vectors in the space of \eqref{rep:induced space 2 k c} as complex-valued functions.
We construct a $(k,c)$ functional on \eqref{rep:induced space 2 k c}, and prove it does not vanish on any of its subrepresentations, in particular on $\rho_c(\tau)$.

Let $v\in V_{(c^k)}$ and set $v_{i,j}=\left(\begin{smallmatrix}v_{i,j}^1&v_{i,j}^2\\v_{i,j}^3&v_{i,j}^4\end{smallmatrix}\right)$, where $v_{i,j}^1\in\Mat_{l}$ and $v_{i,j}^4\in\Mat_{c-l}$. For $t\in\{1,\ldots,4\}$, let $V^t<V_{(c^k)}$ be the subgroup obtained by deleting the blocks $v_{i,j}^{t'}$ for all $i<j$ and $t'\ne t$, and $V=V^3$. Also define
\begin{align*}
&\kappa=\kappa_{l,c-l}=\left(\begin{smallmatrix}I_l\\0&0&I_l\\0&0&0&0&I_l&\ddots\\&&&&&&I_l&0\\0&I_{c-l}\\0&0&0&I_{c-l}&&\ddots\\&&&&&&&I_{c-l}\end{smallmatrix}\right)\in\GL_{kc}.
\end{align*}
\begin{example}
For $c=2$ (then $l=1$) and $k=3$,
\begin{align*}
\kappa=
\left(\begin{smallmatrix}
1\\
&&1\\
&&&&1\\
&1\\
&&&1\\
&&&&&1
\end{smallmatrix}\right),\qquad V=\left\{
\left(\begin{smallmatrix}
1\\
&1&v_{1,2}^3&&v_{1,3}^3\\
&&1\\&&&1&v_{2,3}^3\\&&&&1\\&&&&&1
\end{smallmatrix}\right)\right\}.
\end{align*}
\end{example}
Consider the functional on the space of \eqref{rep:induced space 2 k c},
\begin{align}\label{eq:mnk functional using w}
\xi\mapsto\int\limits_{V}\xi(\kappa v)\,dv.
\end{align}
This is formally a $(k,c)$ functional. Indeed, the conjugation $v\mapsto{}^{\kappa}v$ of $v\in V_{(c^k)}$ takes the blocks $v_{i,j}^1$ onto the subgroup $V_{(l^k)}$ embedded in the top left $kl\times kl$ block of $M_{(kl,k(c-l))}$. Then these blocks transform by the $(k,l)$ functional realizing $W_{\psi}(\rho_l(\tau))$;
$v_{i,j}^2$ is taken to $V_{(kl,k(c-l))}$ and $\xi$ is left-invariant on this group; and after the conjugation, the blocks $v_{i,j}^4$ form the subgroup $V_{((c-l)^k)}$ embedded in the bottom right $k(c-l)\times k(c-l)$ block, and transform by the $(k,c-l)$ functional realizing $W_{\psi}(\rho_{c-l}(\tau))$.
Thus $V_{(c^k)}$ transforms under \eqref{eq:k c character}. Also note that this conjugation takes $v_{i,j}^3$ to $V_{(kl,k(c-l))}^-$ (in particular $V$ is abelian).

For $v\in V$, $y={}^{\kappa}v\in V_{(kl,k(c-l))}^-$ is such that its bottom left $kl\times k(c-l)$ block takes the form
\begin{align}\label{block:V conjugated}
&\left(\begin{smallmatrix}0&v_{1,2}^3&\cdots&v_{1,k}^3\\\vdots&\ddots&\ddots&\vdots\\\vdots&&&v_{k-1,k}^3\\0&\cdots&\cdots&0\end{smallmatrix}\right),\qquad v_{i,j}^3\in\Mat_{(c-l)\times l}.
\end{align}
Let $y_{i,j}\in V_{(kl,k(c-l))}^-$ be obtained from $y$ be zeroing out all the blocks
in \eqref{block:V conjugated} except $v_{i,j}^3$. Also let $X<V_{(kl,k(c-l))}$ be the subgroup of matrices $x$ whose top right $kl\times k(c-l)$ block is
\begin{align*}
&\left(\begin{smallmatrix}
0&0&\cdots&0\\
\vdots&x_{1,2}^3&\ddots&\vdots\\
\vdots&\vdots&\ddots&0\\
0&x_{1,k}^3&\cdots&x_{k-1,k}^3\end{smallmatrix}\right),\qquad x_{i,j}^3\in\Mat_{l\times(c-l)}.
\end{align*}
Define $x_{i,j}$ similarly to $y_{i,j}$. Let $X_{i,j}$ and $Y_{i,j}$ be the respective subgroups of elements. Then
\begin{align}\label{eq:functional xi psi invariance}
\xi(y_{i,j}x_{i,j})=\psi(\mathrm{tr}(v_{i,j}^3x_{i,j}^3))\xi(y_{i,j}).
\end{align}
We show that \eqref{eq:mnk functional using w} can be used to realize $W_{\psi}(\rho_c(\tau))$.
\begin{lemma}\label{lemma:decomposition for V functionals, n=a+b}
For $0<l<c$, realize $\rho_c(\tau)$ as a subrepresentation of
\eqref{rep:induced space 2 k c}. The integral
\eqref{eq:mnk functional using w} is absolutely convergent, and is a (nonzero) $(k,c)$ functional
on $\rho_c(\tau)$.
\end{lemma}
%Using induction, this lemma can also be used to deduce that $\rho_c(\tau)$ admits a $(k,c)$ functional, the base case $c=1$ holding because %$\tau$ is irreducible generic.
\begin{proof}
The proof technique is called ``root elimination", see e.g., \cite[Proposition~6.1]{Soudry}, \cite[\S~7.2]{Soudry} and \cite[\S~6.1]{Jac5} (also the proof of \cite[Lemma~8]{LR}).
We argue by eliminating each $y_{i,j}$ separately, handling the diagonals left to right, bottom to top: starting with $y_{k-1,k}$, next $y_{k-2,k-1}$, ..., up to $y_{1,2}$, then $y_{k-2,k}$, $y_{k-3,k-1}$, etc., with $y_{1,k}$ handled last.
Let $\mathcal{W}$ be a subrepresentation of \eqref{rep:induced space 2 k c}. For $\xi$ in the space of $\mathcal{W}$ and a Schwartz function $\phi$ on $\Mat_{l\times(c-l)}$ (over $p$-adic fields, Schwartz functions are in particular compactly supported), define
\begin{align*}
&\xi_{i,j}(g)=\int_{X_{i,j}}\xi(gx_{i,j})\phi(x_{i,j}^3)\,dx_{i,j},\\
&\xi'_{i,j}(g)=\int_{Y_{i,j}}\xi(gy_{i,j})\widehat{\phi}(v_{i,j}^3)\,dy_{i,j}.
\end{align*}
Here $\widehat{\phi}$ is the Fourier transform of $\phi$ with respect to $\psi\circ\mathrm{tr}$. By smoothness over $p$-adic fields, or by the
Dixmier--Malliavin Theorem \cite{DM}
over archimedean fields, any $\xi$ is a linear combination of functions
$\xi_{i,j}$, and also of functions $\xi'_{i,j}$. Using \eqref{eq:functional xi psi invariance}, the definition of the Fourier
transform and the fact that $V$ is abelian we obtain, for $(i,j)=(k-1,k)$,
\begin{align*}
\int_{Y_{i,j}}\xi_{i,j}(y_{i,j}{}^{\kappa}v^{\circ})\,dy_{i,j}=\xi'_{i,j}({}^{\kappa}v^{\circ}),
\end{align*}
where $v^{\circ}\in V$ does not contain the block of $v_{i,j}^3$. We re-denote $\xi=\xi'_{i,j}$, then proceed similarly
with $(i,j)=(k-2,k-1)$. This shows that the integrand is a Schwartz function on ${}^{\kappa}V$, thus
the integral \eqref{eq:mnk functional using w} is absolutely convergent. At the same time, the integral does
not vanish on $\mathcal{W}$ because in this process we can obtain $\xi(I_{kc})$. Over archimedean fields the same argument also implies
that \eqref{eq:mnk functional using w} is continuous (see \cite[\S~5, Lemma~2, p.~199]{Soudry3}).

Over $p$-adic fields we provide a second argument for the nonvanishing part. Choose $\xi_0$ in the space of $\mathcal{W}$ with $\xi_0(I_{kc})\ne0$. Define for a (large) compact subgroup $\mathcal{X}<X$, the function
\begin{align*}
\xi_1(g)=\int\limits_{\mathcal{X}}\xi(gx)\,dx,
\end{align*}
which clearly also belongs to the space of $\mathcal{W}$. We show that for a sufficiently large $\mathcal{X}$,
$\int\limits_{V}\xi_1({}^{\kappa}v)\,dv=\xi(I_{kc})$. Put $y={}^{\kappa}v$. We prove $\xi_1(y)=0$, unless $y$ belongs to a small compact neighborhood of the identity, and then
$\xi_1(y)=\xi(I_{kc})$. We argue by eliminating each $y_{i,j}$ separately, in the order stated above. Assume we have zeroed out all blocks on the diagonals to the left of $y_{i,j}$, and below $y_{i,j}$ on its diagonal. Let $\mathcal{B}$ denote the set of indices $(i',j')$ of the remaining $y_{i',j'}$ and $\mathcal{B}^{\circ}=\mathcal{B}-(i,j)$. Denote $\mathcal{X}_{i,j}=\mathcal{X}\cap X_{i,j}$ and assume that if $(i',j')\notin\mathcal{B}$, $\mathcal{X}_{i',j'}$ is trivial. Write $\mathcal{X}=\mathcal{X}^{\circ}\times \mathcal{X}_{i,j}$. For $(i',j')\in\mathcal{B}^{\circ}$, ${}^{y_{i',j'}}x_{i,j}=ux_{i,j}$, where $u\in V_{(l^k)}\times V_{((c-l)^k)}$ and the $(k,l)$ and $(k,c-l)$ characters
\eqref{eq:k c character} are trivial on $u$. Therefore by \eqref{eq:functional xi psi invariance},
\begin{align*}
\xi_1(y)=\int\limits_{\mathcal{X}^{\circ}}\xi(yx)\,dx\int\limits_{\mathcal{X}_{i,j}}\psi(\mathrm{tr}(v_{i,j}^3x_{i,j}^3))\,dx_{i,j}.
\end{align*}
The second integral vanishes unless $v_{i,j}^3$ is sufficiently small, then this integral becomes a nonzero measure constant. Moreover,
if $\mathcal{X}_{i,j}$ is sufficiently large with respect to $\xi$ and $\mathcal{X}_{i',j'}$ for all $(i',j')\in \mathcal{B}^{\circ}$,
then for any $x\in\mathcal{X}^{\circ}$, ${}^{y_{i,j}}x=xz$ where $z$ belongs to a small neighborhood of the identity in $\GL_{kc}$, on which $\xi$ is invariant on the right. Therefore we may remove $y_{i,j}$ from $y$ in the first integral.
Re-denote $\mathcal{X}=\mathcal{X}^{\circ}$. Repeating this process, the last step is for $y_{1,k}$ with an integral over $\mathcal{X}_{1,k}$, and we remain with $\xi(I_{kc})$.
\end{proof}

In contrast with the realization described in the previous section, here we regard $\rho_c(\tau)$ as a subrepresentation, therefore nonvanishing on $\rho_c(\tau)$ is not a consequence of uniqueness. In fact, the proof above implies
$(k,c)$ functionals on the space of \eqref{rep:induced space 2 k c} are not necessarily unique. On the other hand \eqref{eq:mnk functional using w} is already absolutely convergent (we do not need twists).

\subsection{Certain unramified principal series}\label{k c functionals unr principal}
One may consider integral~\eqref{eq:(k,c)  functional using an integral} (now with any $d\geq2$) also when $\tau=\Ind_{B_{\GL_k}}^{\GL_k}(\otimes_{i=1}^k\tau_i)$ is an unramified principal series (possibly reducible), written with a weakly decreasing order of exponents (see before Theorem~\ref{theorem:any rep is WSS}). Then $\rho_c(\tau)=\Ind_{P_{(c^k)}}^{\GL_{kc}}(\otimes_{i=1}^k\tau_i\circ\det)$. The integral admits analytic continuation in the twisting parameters $\zeta_1,\ldots,\zeta_k$, which is nonzero for all choices of $\zeta_i$. Hence it defines a $(k,c)$ functional on $\rho_c(\tau)$. In this setting we can prove a decomposition result similar to
Lemma~\ref{lemma:decomposition for V functionals, n=a+b} but using \eqref{eq:(k,c)  functional using an integral}.
The results of this section are essentially an elaborative reformulation of \cite[Lemma~22]{CFGK2}.

Put $\zeta=(\zeta_1,\ldots,\zeta_k)\in\C^k$.
Denote the representation $\Ind_{P_{(c^k)}}^{\GL_{kc}}(\otimes_{i=1}^k|~|^{\zeta_i}\tau_i\circ\det)$ by $\rho_c(\tau_{\zeta})$, with a minor abuse of notation (because we do not consider only $\zeta\in\R^k$ and $\zeta_1\geq\ldots\geq\zeta_k$).
Let $V(\zeta,\tau,c)$ be the space of $\rho_c(\tau_{\zeta})$. A section $\xi$ on $V(\tau,c)$ is a function $\xi:\C^k\times\GL_{kc}\rightarrow\C$ such that for all $\zeta\in\C^k$,
$\xi(\zeta,\cdot)\in V(\zeta,\tau,c)$, and we call it entire if
$\zeta\mapsto\xi(\zeta,g)\in\C[q^{\pm\zeta_1},\ldots,q^{\pm\zeta_k}]$, for all $g\in\GL_{kc}$. A meromorphic section is a function $\xi$ on
$\C^k\times\GL_{kc}$ such that $\varphi(\zeta)\xi(\zeta,g)$ is an entire section, for some holomorphic and not identically zero
$\varphi:\C^k\rightarrow\C$. The normalized unramified section $\xi^0$ is the section which is the normalized unramified vector for all $\zeta$.

Let $0<l<c$. We use the notation $\kappa$, $v_{i,j}$, $v_{i,j}^t$, $t\in\{1,\ldots,4\}$, $V^t$ and $V=V^3$ from the previous section.
Denote $Z={}^{(\diag(w_{(l^k)},w_{((c-l)^k)})\kappa)}V^2$. Define an intertwining operator
by the meromorphic continuation of the integral
\begin{align*}
m(\zeta,\kappa)\xi(\zeta,g)=\int\limits_{Z}\xi(\zeta,\kappa^{-1}zg)\,dz,
\end{align*}
where $\xi$ is a meromorphic section of $V(\tau,c)$. When $\xi$ is entire, this integral is absolutely convergent for $\Real(\zeta)$ in a cone
of the form $\Real(\zeta_1)\gg\ldots\gg\Real(\zeta_k)$, which depends only on the inducing characters.
\begin{lemma}\label{lemma:m zeta tau explicit}
Assume $1-q^{-s}\tau_i(\varpi)\tau_j^{-1}(\varpi)$ is nonzero for all $i<j$ and $\Real(s)\geq1$.
Then for all $\zeta$ with $\Real(\zeta_1)\geq\ldots\geq\Real(\zeta_k)$, $m(\zeta,\kappa)\xi^0(\zeta,\cdot)$ is well defined, nonzero and belongs to the space of
\begin{align}\label{rep:zeta unr rho c tau}
\Ind_{P_{(kl,k(c-l))}}^{\GL_{kc}}((\rho_l(\tau_{\zeta})\otimes\rho_{c-l}(\tau_{\zeta}))\delta_{P_{(kl,k(c-l))}}^{-1/(2k)}).
\end{align}
\end{lemma}
\begin{proof}
The trivial representation of $\GL_c$ is an unramified subrepresentation of
$\Ind_{B_{\GL_c}}^{\GL_c}(\delta_{B_{\GL_c}}^{-1/2})$, hence $\rho_c(\tau_{\zeta})$ is an unramified subrepresentation of
\begin{align*}
\Ind_{P_{(c^k)}}^{\GL_{kc}}(\otimes_{i=1}^k\Ind_{B_{\GL_c}}^{\GL_c}(|~|^{\zeta_i}\tau_i\delta_{B_{\GL_c}}^{-1/2}))
=\Ind_{B_{\GL_{kc}}}^{\GL_{kc}}(\otimes_{i=1}^k|~|^{\zeta_i}\tau_i\delta_{B_{\GL_c}}^{-1/2}).
\end{align*}
Looking at $\kappa$, we see that the image of $m(\zeta,\kappa)$ on the space of this representation is contained in
\begin{align}\nonumber
&\Ind_{B_{\GL_{kc}}}^{\GL_{kc}}(|\det|^{-(c-l)/2}(\otimes_{i=1}^k|~|^{\zeta_i}\tau_i\delta_{B_{\GL_l}}^{-1/2})
\,\otimes\,|\det|^{l/2}(\otimes_{i=1}^k|~|^{\zeta_i}\tau_i\delta_{B_{\GL_{c-l}}}^{-1/2}))\\
&=\label{rep:image of m zeta kappa}
\Ind_{P_{(kl,k(c-l))}}^{\GL_{kc}}(\big(\Ind_{B_{\GL_{kl}}}^{\GL_{kl}}(\otimes_{i=1}^k|~|^{\zeta_i}\tau_i\delta_{B_{\GL_l}}^{-1/2})
\,\otimes\,\Ind_{B_{\GL_{k(c-l)}}}^{\GL_{k(c-l)}}(\otimes_{i=1}^k|~|^{\zeta_i}\tau_i\delta_{B_{\GL_{c-l}}}^{-1/2})\big)\delta_{P_{(kl,k(c-l))}}^{-1/(2k)}).
\end{align}
This representation contains \eqref{rep:zeta unr rho c tau} as an unramified subrepresentation.
We show $m(\zeta,\kappa)\xi^0(\zeta,\cdot)$ satisfies the required properties for the prescribed $\zeta$.
We may decompose $m(\zeta,\kappa)$ into rank-$1$ intertwining operators on spaces of the form
\begin{align*}
\Ind_{B_{\GL_2}}^{\GL_2}(|~|^{\zeta_i-(c-2l+1)/2}\tau_i\otimes|~|^{\zeta_j-(c-2l'+1)/2}\tau_j),\qquad i<j,\qquad l'\leq l-1.
\end{align*}
According to the Gindikin--Karpelevich formula (\cite[Theorem~3.1]{CS1}), each intertwining operator takes the normalized unramified vector in this space to a constant multiple of the normalized unramified vector in its image, and this constant is given by
\begin{align*}
\frac{1-q^{-1-\zeta_i+\zeta_j-l+l'}\tau_i(\varpi)\tau_j^{-1}(\varpi)}{1-q^{-\zeta_i+\zeta_j-l+l'}\tau_i(\varpi)\tau_j^{-1}(\varpi)}.
\end{align*}
Since $\Real(-\zeta_i+\zeta_j)\leq 0$ and $-l+l'\leq-1$, if the quotient has a zero or pole, then
$1-q^{-s}\tau_i(\varpi)\tau_j^{-1}(\varpi)=0$ for $\Real(s)\geq1$, contradicting our assumption. Therefore $m(\zeta,\kappa)\xi^0(\zeta,\cdot)$ is well defined and nonzero, and because it is unramified, it also belongs to
\eqref{rep:zeta unr rho c tau}.
\end{proof}
Integral~\eqref{eq:(k,c)  functional using an integral} is also absolutely convergent for $\Real(\zeta)$ in a cone
$\Real(\zeta_1)\gg\ldots\gg\Real(\zeta_k)$, which depends only on the inducing characters. The proof is that of the known result
for similar intertwining integrals.
\begin{lemma}\label{lemma:decomposition lemma k c func reducible unramified}
In the domain of absolute convergence of \eqref{eq:(k,c)  functional using an integral} and in general by meromorphic continuation,
for any meromorphic section $\xi$ on $V(\tau,c)$,
\begin{align*}
\int\limits_{V_{(c^k)}}\xi(\zeta,w_{(c^k)}v)\psi^{-1}(v)\,dv=
\int\limits_{V}m(\zeta,\kappa)\xi(\zeta,\kappa v)\,dv.
\end{align*}
Here $m(\zeta,\kappa)\xi$ belongs to the space obtained from \eqref{rep:image of m zeta kappa} by applying the
$(k,l)$ and $(k,c-l)$ functionals \eqref{eq:(k,c)  functional using an integral} on the respective factors of $P_{(kl,k(c-l))}$.
\end{lemma}
\begin{proof}
Using matrix multiplication we see that $w_{(c^k)}=\kappa^{-1}\diag(w_{(l^k)},w_{((c-l)^k)})\kappa$. The character $\psi$ is trivial on $V^2$. Thus in its domain of absolute convergence integral~\eqref{eq:(k,c)  functional using an integral} equals
\begin{align*}
\int\limits_{V}\int\limits_{V^4}\int\limits_{V^1}m(\zeta,\kappa)\xi(\zeta,\diag(w_{(l^k)}{}^{\kappa}v_1,w_{((c-l)^k)}{}^{\kappa}v_4)\kappa v)\psi^{-1}(v_1)\psi^{-1}(v_4)\,dv_1\,dv_4\,dv.
\end{align*}
The integrals $dv_1dv_4$ constitute the applications of $(k,l)$ and $(k,c-l)$ functionals, e.g., ${}^{\kappa}V_1=\diag(V_{(l^k)},I_{k(c-l)})$ (see after \eqref{eq:mnk functional using w}). Now combine this with
the proof of Lemma~\ref{lemma:m zeta tau explicit}.
\end{proof}
Combining this result for $\xi^0$ with Lemma~\ref{lemma:m zeta tau explicit}, we obtain a result analogous to Lemma~\ref{lemma:decomposition for V functionals, n=a+b}.

\subsection{Equivariance property under $\GL_c^{\triangle}$}\label{k c functionals additional inv}
The following additional equivariance property of $(k,c)$ functionals is crucial for constructing integrals
involving $(k,c)$ models. Let $g\mapsto g^{\triangle}$ be the diagonal embedding of $\GL_c$ in $\GL_{kc}$.
\begin{lemma}\label{lemma:left invariance of k c on GL}
Let $\lambda$ be a $(k,c)$ functional on $\rho_c(\tau)$ and $\xi$ be a vector in the space of $\rho_c(\tau)$.
For any $g\in\GL_c$, $\lambda(\rho_c(\tau)(g^{\triangle})\xi)=\tau(\det(g)I_k)\lambda(\xi)$.
\end{lemma}
\begin{proof}
The claim clearly holds for $c=1$, since then $g^{\triangle}$ belongs to the center of $\GL_k$. Let $c>1$.
We prove separately that $\lambda(\rho_c(\tau)(t^{\triangle})\xi)=\tau(\det(t))\lambda(\xi)$ for all $t\in T_{\GL_c}$ and
$\lambda(\rho_c(\tau)(g^{\triangle})\xi)=\lambda(\xi)$ for all $g\in \SL_c$. Since all $(k,c)$ functionals are proportional, we may prove each equivariance property using a particular choice of functional.

First take $t=\diag(t_1,\ldots,t_c)$.
Assume $\tau$ is irreducible essentially tempered. Consider the $(k,c)$ functional
\eqref{eq:mnk functional using w} with $l=1<c$. Conjugate $V$ by $t^{\triangle}$, then
${}^{\kappa}(t^{\triangle})=\diag(t_1I_k,{t'}^{\triangle})$ with $t'=\diag(t_2,\ldots,t_c)$ and
${t'}^{\triangle}\in \GL_{k(c-1)}$. The change to the measure of $V$ is
$\delta_{P_{(k,k(c-1))}}^{-1/2+1/(2k)}({}^{\kappa}t)$ and the result now follows using induction. The case of irreducible generic $\tau$ is reduced to the essentially tempered case using \eqref{eq:(k,c)  functional using an integral} and note that $\delta_{P_{(\beta c)}}(t^{\triangle})=1$ for any composition $\beta$ of $k$. If $\tau$ is
a reducible unramified principal series we again compute using \eqref{eq:(k,c)  functional using an integral}.

It remains to consider $g\in\SL_c$. By definition the Jacquet module of $\rho_c(\tau)$ with respect to $V_{(c^k)}$ and \eqref{eq:k c character} is one dimensional (whether $\tau$ is irreducible or unramified principal series), hence $\SL_c^{\triangle}$ acts trivially on the Jacquet module, and the result follows.
\end{proof}

\def\cprime{$'$} \def\cprime{$'$} \def\cprime{$'$}

\end{document}